\documentclass[leqno]{amsart}
\usepackage[utf8]{inputenc}

\usepackage{amsmath}
\usepackage{amssymb}
\usepackage{amsthm}
\usepackage{amsfonts}
\usepackage{enumerate}
\usepackage{mathtools}

\usepackage[colorlinks=false]{hyperref}

\usepackage{xy}
\xyoption{all}

\numberwithin{equation}{section}

\theoremstyle{plain}
    \newtheorem{theorem}[equation]{Theorem}
    \newtheorem{lemma}[equation]{Lemma}
    \newtheorem{corollary}[equation]{Corollary}

    \newtheorem*{theorem*}{Theorem}
    \newtheorem*{proposition*}{Proposition}
    \newtheorem*{corollary*}{Corollary}
    \newtheorem*{lemma*}{Lemma}
    \newtheorem*{conjecture*}{Conjecture}
    \newtheorem{definition-theorem}[equation]{Definition/Theorem}
    \newtheorem{definition-lemma}[equation]{Definition/Lemma}
\theoremstyle{definition}
    \newtheorem{definition}[equation]{Definition}
    \newtheorem{example}[equation]{Example}

    \newtheorem{remark}[equation]{Remark}

    \newcommand{\C}{\mathbb{C}}

   	\renewcommand{\phi}{\varphi}
	\let\epsilon\varepsilon

    \newcommand{\Bounded}{\operatorname{B}}
    \newcommand{\CB}{\operatorname{CB}}
    \newcommand{\Compact}{\operatorname{K}}
    \newcommand{\Adjointable}{\operatorname{L}}
    \newcommand{\cb}{\mathrm{cb}}
    \newcommand{\alg}{\mathrm{alg}}
    \newcommand{\functor}{\mathcal}
    \newcommand{\category}{\mathsf}

\newcommand{\id}{\mathrm{id}}
\newcommand{\ol}{\overline}

    \DeclareMathOperator{\Hom}{Hom}

    \DeclareMathOperator{\Res}{Res}
    \DeclareMathOperator{\Ind}{Ind}

    \DeclareMathOperator{\GL}{GL}

	\DeclareMathOperator{\opmod}{OM}
	
	\DeclareMathOperator{\URep}{URep}
	\DeclareMathOperator{\hmod}{HM}
	\DeclareMathOperator{\cmod}{CM}

\title{Frobenius reciprocity and the Haagerup tensor product}
\author{Tyrone Crisp }
\address{Max Planck Institute for Mathematics, Vivatsgasse 7, 53111 Bonn, Germany. }
\address{Department of Mathematics, Radboud University Nijmegen, PO Box 9010, 6500GL  Nijmegen, The Netherlands {\normalfont (current address)}}
\email{t.crisp@math.ru.nl}

\date{July, 2016. Revised February, 2017.}

\begin{document}

\begin{abstract}
In the context of operator-space modules over $C^*$-algebras, we give a complete characterisation of those $C^*$-correspondences whose associated Haagerup tensor product functors admit left adjoints. The characterisation, which builds on previous joint work with N.~Higson, exhibits a close connection between the notions of adjoint operators and adjoint functors. As an application, we prove a Frobenius reciprocity theorem for representations of locally compact groups on operator spaces: the functor of unitary induction for a closed subgroup $H$ of a locally compact group $G$ admits a left adjoint in this setting if and only if $H$ is cocompact in $G$. The adjoint functor is given by Haagerup tensor product with the operator-theoretic adjoint of Rieffel's induction bimodule.
\end{abstract}


\maketitle

\section{Introduction}

The reciprocity between induced and restricted group characters  discovered by Frobenius is a cornerstone in the representation theory of finite groups, establishing a useful explicit relationship between the representations of a group and the representations of its subgroups. The computation of the character table of the symmetric groups by Frobenius himself is an early example of the reciprocity law's many applications. See \cite{Curtis} for an account of Frobenius's computation. Turning from finite-dimensional representations of finite groups to infinite-dimensional representations of locally compact groups, Frobenius reciprocity was also a central concern of Mackey in his influential work on induced unitary representations, and various forms of the reciprocity law have been established in this context. See  \cite{Weil},  \cite{Mackey_inducedI}, \cite{Mautner}, \cite{Mackey_inducedII}, \cite{Fell} and \cite{Rosenberg}, for example.

In its strongest formulation, Frobenius reciprocity is the assertion that for a closed subgroup $H$ of a locally compact group $G$, Mackey's unitary induction functor 
\[
\Ind_H^G: \URep(H)\to \URep(G)
\]
from unitary representations of $H$ to unitary representations $G$ is \emph{adjoint}, on both sides, to the restriction functor $\Res^G_H$. That is, there are natural isomorphisms between the spaces of equivariant bounded operators
\begin{equation*}\label{eq:intro_frob}
\Bounded_G(K,\Ind_H^G L) \cong \Bounded_H(\Res^G_H K, L) \quad \text{and}\quad \Bounded_G(\Ind_H^G L, K) \cong \Bounded_H(L, \Res^G_H K)
\end{equation*}
for all unitary representations $K\in \URep(G)$ and $L\in\URep(H)$. 

This formulation is easily seen to be false, in general; indeed, it frequently happens that neither $\Ind_H^G$ nor $\Res^G_H$ admits any adjoint at all. This is the case, for instance, when $G$ is infinite (e.g. compact and non-discrete) and $H$ is the trivial subgroup, and it is instructive to see why.

Suppose, say, that the induction functor $\Ind_{1}^G$, which sends a Hilbert space $L$ to the $G$-representation $L^2(G)\otimes L$, had a left adjoint $\functor{L}:\URep(G)\to \URep(1)$. Then we would have in particular an isomorphism 
\[
\Bounded_{\C}(\functor{L}L^2(G),\C ) \cong \Bounded_G(L^2(G), L^2(G)).
\]
This is, a priori, an isomorphism of sets, but the continuity of the induction functor  and the naturality of the isomorphisms ensure that it is in fact an isomorphism of Banach spaces. Thus  $\functor{L}L^2(G)$ is isomorphic, as a Banach space, to the Fourier algebra $A(G)$ (the predual of the von Neumann algebra of $G$; see \cite{Eymard}). In particular, $\functor{L}L^2(G)$ is not a Hilbert space.

This example suggests that  if we are to insist on the adjoint functor formulation of Frobenius reciprocity---and there are good reasons for doing so; see below---then we must enlarge the category $\URep$ to include representations on more general Banach spaces. One option is to consider representations on arbitrary Banach spaces, with Mackey's unitary induction functor replaced by a suitable $L^1$ variant. Very satisfactory results in this context were obtained in \cite{Moore} and \cite{Rieffel_Banach}. 

The category of all Banach representations of $G$ is generally much larger than $\URep(G)$. In this paper we shall consider instead the category of \emph{operator modules} over the group $C^*$-algebra $C^*(G)$, and over $C^*$-algebras more generally. This category provides a natural venue for adjoint functor techniques, while remaining close enough to the Hilbert space setting so that theorems about operator modules yield theorems about Hilbert space representations.

Concretely, a (left) operator module over  $C^*(G)$  is a norm-closed linear subspace  $X\subseteq\Bounded(K)$, where $K$ is a unitary representation of $G$, such that $g\circ x\in X$ for every $g\in G$ and every $x\in X$. See \cite{BLM} for details, or Section \ref{sec:prelims} for a brief review. The category $\opmod(A)$ of operator modules over an arbitrary $C^*$-algebra $A$ contains as a subcategory the category $\hmod(A)$ of $*$-representations of $A$ on Hilbert space; in particular, $\opmod(C^*(G))$ contains $\URep(G)$. The category $\opmod(A)$ also contains the category $\cmod(A)$ of (left) Hilbert $C^*$-modules over $A$, so  for instance  $A$ itself is an object in $\opmod(A)$. A theorem of Blecher \cite{Blecher} asserts that $\opmod(A)$ remembers both of these subcategories: every  suitably continuous  categorical equivalence $\opmod(A)\cong\opmod(B)$ comes from a (strong) Morita equivalence between $A$ and $B$, and hence restricts to equivalences $\hmod(A)\cong\hmod(B)$ and $\cmod(A)\cong\cmod(B)$. Another result along these lines, proved in \cite{Crisp-Higson}, asserts that if $A$ is a type-$\mathrm{I}$ $C^*$-algebra, then $\hmod(A)$ and $\opmod(A)$ contain the same irreducible objects.

We shall study functors of the form 
\begin{equation}\label{eq:intro_F}
\functor{F}:\opmod(B)\to \opmod(A) \qquad X\mapsto F\otimes^h_B X
\end{equation}
where $A$ and $B$ are $C^*$-algebras,   $F$ is a \emph{$C^*$-correspondence} from   $A$ to  $B$, and $\otimes^h$ denotes the \emph{Haagerup tensor product} (the terminology will be explained in Section \ref{sec:prelims}). Blecher has shown that this functor $\functor{F}$, restricted to the subcategories $\hmod$ of Hilbert-space representations, coincides with the functor of interior $C^*$-module tensor product, as studied by Rieffel \cite{Rieffel_induced}. 

Rieffel showed that our motivating example, unitary induction of group representations, is a functor of this sort. As this example demonstrates, the tensor product functors associated to $C^*$-correspondences frequently do not have adjoints in the setting of Hilbert space representations. But in the operator module setting, functors of the form \eqref{eq:intro_F} always have `operator-theoretic adjoints': the $A$-$B$ bimodule $F$ can be concretely represented as a set of Hilbert space operators, and each operator $f\in F$ has a uniquely defined adjoint operator $f^*$ characterised by the usual equation:
\begin{equation}\label{eq:intro_opadj}
\langle y| fx \rangle = \langle f^* y |x\rangle.
\end{equation}
The set $F^*=\{f^*\ |\ f\in F\}$ is an operator $B$-$A$ bimodule, and gives rise to a Haagerup tensor product functor 
\[
\functor{F}^*:\opmod(A)\to \opmod(B) \qquad Y\mapsto F^*\otimes^h_A Y.
\]
The functor $\functor{F}^*$ sends Hilbert $C^*$-modules to Hilbert $C^*$-modules, but it does not necessarily send Hilbert spaces to Hilbert spaces. 

In \cite{Crisp-Higson}  it was shown that if the action of $A$ on $F$ is by compact operators (in the sense of Hilbert $C^*$-modules), then the operator-theoretic adjoint functor $\functor{F}^*$ is an actual left adjoint to $\functor{F}$ in the sense of category theory: there are natural isomorphisms between the spaces of completely bounded module maps
\begin{equation}\label{eq:intro_adjunction}
\CB_A(Y,\functor{F}X) \cong \CB_B(\functor{F}^* Y, X)
\end{equation}
for all $X\in \opmod(B)$ and all $Y\in \opmod(A)$. The main result of this paper is a strong converse to this assertion: 

{
\renewcommand{\theequation}{\ref{thm:adj_main}}
\begin{theorem}
 If $\functor{F}$ admits any left adjoint at all, then that adjoint is $\functor{F}^*$, and $A$ acts on $F$ by compact operators.
\end{theorem}
\addtocounter{equation}{-1}
} 
The theorem shows that in the context of operator modules there is a close and explicit connection between the notions of adjoint operators and adjoint functors, going beyond the obvious notational coincidence between \eqref{eq:intro_opadj} and \eqref{eq:intro_adjunction}. See \cite{MacLane_Stone} and \cite{Palmquist} for other parallels between the two notions, in other settings. 

The question of when $\functor{F}$ admits a \emph{right} adjoint remains open.  This is perhaps surprising, since in purely algebraic situations tensor product functors  always have right adjoints, given by $\Hom$. But in functional-analytic settings one is often faced with a multitude of different tensor products, and the lack of an internal $\Hom$ functor---for instance, the space of bounded linear maps between two Hilbert spaces is seldom itself a Hilbert space---and this makes the `standard' $\otimes$-$\Hom$ adjunction a rather delicate matter. See \cite[Proposition 3.5.9]{BLM} for an example involving the \emph{projective} tensor product of operator spaces; this tensor product does not, in general, send operator modules to operator modules.

Several corollaries to the main theorem are presented in Section \ref{sec:cor}. For example:

{
\renewcommand{\theequation}{\ref{cor:unitary_frob}}
\begin{corollary}
For a locally compact group $G$ with closed subgroup $H$, the unitary induction functor
\[
\Ind_H^G:\opmod(C^*(H)) \to \opmod(C^*(G))
\]
has a left adjoint if and only if $H$ is cocompact in $G$, while the restriction functor
\[
\Res^G_H:\opmod(C^*(G))\to \opmod(C^*(H))
\]
has a left adjoint if and only if $H$ is open in $G$.
\end{corollary}
\addtocounter{equation}{-1}
}
To return to  the example $\Ind_1^G$ of induction from the trivial subgroup, this functor has a left adjoint in the operator module setting if and only if $G$ is compact. As predicted, the adjoint $(\Ind_1^G)^*$ sends the regular representation $L^2(G)$ to the Fourier algebra of $G$. This last assertion is in fact true regardless of whether $G$ is compact, and it suggests that the operator-adjoint functors $(\Ind_H^G)^*$ may have a useful role to play in  harmonic analysis. 
 
A second corollary of Theorem \ref{thm:adj_main} concerns the relationship between the categories $\opmod$ and $\hmod$ as regards adjoint correspondences.
{
\renewcommand{\theequation}{\ref{cor:adjcorresp}}
\begin{corollary}
Consider the Haagerup tensor product functors $\functor{F}:\opmod(B)\to\opmod(A)$ and $\functor{E}:\opmod(A)\to\opmod(B)$ associated to a pair of $C^*$-correspondences. If one of these functors is left-adjoint to the other, then the restricted functors $\functor{F}_{\hmod}:\hmod(B)\to\hmod(A)$ and $\functor{E}_{\hmod}:\hmod(A)\to\hmod(B)$ are two-sided adjoints. 
\end{corollary}
\addtocounter{equation}{-1}
}

Our interest in adjoint functor techniques for unitary representations arose from a study of parabolic induction for representations of real reductive groups. In \cite{CCH1} and \cite{CCH2} we showed that the parabolic induction functor $\Ind_P^G:\URep(L)\to \URep(G)$ associated to a Levi factor $L$ of a parabolic subgroup $P$ of a real reductive group $G$ admits a two-sided adjoint. This is reminiscent of a theorem of Bernstein \cite{Bernstein}, which asserts that in the related setting of smooth representations of $p$-adic reductive groups, parabolic induction admits both a left and a right adjoint. In \cite{Crisp-Higson} we showed that the natural extension of $\Ind_P^G$ to a functor on operator modules possesses a left adjoint. In view of Bernstein's theorem, and of the results of \cite{CCH1} and \cite{CCH2}, it is natural to ask whether that same functor also admits a right adjoint.
{
\renewcommand{\theequation}{\ref{cor:parabolic}}
\begin{corollary}
If $G$ is a real reductive group, and $L$ is a Levi factor of a proper parabolic subgroup $P$ of $G$, then the parabolic induction functor 
\[
\Ind_P^G:\opmod(C^*_r(L))\to \opmod(C^*_r(G))
\]
studied in \cite{Crisp-Higson} does not have a strongly continuous right adjoint. 
\end{corollary}
\addtocounter{equation}{-1}
}
In particular, there is no `second adjoint theorem', in the spirit of \cite{Bernstein}, for operator modules over $C^*$-algebras. (Cf.~\cite[Section 4.3]{Crisp-Higson} for a second adjoint theorem involving non-self-adjoint operator algebras, and  \cite{CH-HC} for a second adjoint theorem in the setting of Fr\'echet modules.)

Let us conclude this introduction with a few words on our motivation for insisting on adjoint functors. If $\functor{F}:\category{A}\to\category{B}$ is a functor with a left adjoint $\functor{F}^*:\category{B}\to\category{A}$, then the composite functor $\functor{M}=\functor{F}\functor{F}^*$ on $\category{B}$ is a \emph{monad} (see \cite[Chapter VI]{MacLane}). Various \emph{monadicity theorems} assert that under favourable conditions, the image of $\functor{F}$ can be characterised as precisely those objects of $\category{B}$ which can be equipped with   a module structure   over $\functor{M}$, and  that the category of $\functor{M}$-modules is equivalent to $\category{A}$.  

This calls to mind the imprimitivity theorem of Mackey \cite{Mackey_imprimitivity} and its many $C^*$-algebraic generalisations, beginning with the work of Rieffel \cite{Rieffel_induced}. The imprimitivity theorem characterises the image of the unitary induction functor $\Ind_H^G$ as those representations of $G$ which are modules over the crossed product $C_0(G/H)\rtimes G$, and asserts moreover that this crossed product is Morita equivalent to $C^*(H)$. Indeed, Mackey's theorem can be interpreted as a monadicity theorem for the functor $\Ind_H^G$ in the operator module setting. It is interesting to note, though,  that Mackey's theorem and its generalisations are valid even in cases where one does not have an adjoint functor in the sense of category theory: the operator-theoretic adjoint $\functor{F}^*$ is enough. These matters, and the so far unexplored \emph{co}monadic counterpart to Rieffel's theory, will be taken up in a future work.

\section{Preliminaries from operator theory and category theory}\label{sec:prelims}

In this section we shall review some basic notions regarding operator modules and adjoint functors, referring   the reader to the books \cite{BLM} and \cite{MacLane} for more information (and for references to the primary sources).

\subsection{Operator modules over $\boldsymbol {C^*}$-algebras}\label{subsec:opmod}

An \emph{operator space} is a norm-closed linear space $X\subseteq \Bounded(H)$ of bounded operators on some complex Hilbert space. For each $n\geq 1$ the operator norm on $M_n(\Bounded(H))\cong \Bounded(H^n)$ restricts to a norm $\|\cdot\|_n$ on $M_n(X)$. A linear map between operator spaces $t:X\to Y$ is \emph{completely bounded} if the maps $t_n:M_n(X)\to M_n(Y)$ defined by applying $t$ entrywise are bounded uniformly in $n$: i.e., if the norm $\|t\|_{\cb}\coloneqq \sup_{n} \| t_n\|$
is finite. Such a map $t$ is \emph{completely contractive} if $\|t\|_{\cb}\leq 1$, and $t$ is \emph{completely isometric} if each $t_n$ is an isometry.   The space $\CB(X,Y)$ of completely bounded maps from $X$ to $Y$ is  a Banach space in the norm $\|\cdot\|_{\cb}$ (it is in fact an operator space in its own right). An isomorphism of operator spaces means, for us, a completely bounded linear isomorphism with completely bounded inverse.

Let $A$ be a $C^*$-algebra, and let $X$ be an operator space which is also a left $A$-module. We say that $X$ is an \emph{operator module} over $A$ if for every $n\geq 1$, the action of the $C^*$-algebra $M_n(A)$ on $M_n(X)$ by matrix multiplication satisfies 
\[
\| ax\|_n \leq \|a\|\cdot \|x\|_n.
\]
The concepts of right operator module, and operator bimodule, are defined analogously. 

For any operator (or Banach) module $X$ over $A$, the \emph{nondegenerate part} of $X$ is defined as 
\[
A\cdot X \coloneqq \{ax\ |\ a\in A,\ x\in X\}.
\]
This is a closed $A$-submodule of $X$, by the Cohen factorisation theorem. We say that $X$ is \emph{nondegenerate} if $A\cdot X=X$. 

We denote by $\opmod(A)$ the category whose objects are the  nondegenerate left operator $A$-modules, with morphisms $\CB_A(X,Y)$ the completely bounded, $A$-linear maps.

\subsection{Adjoint operator modules}
If $X$ is an operator $A$-$B$ bimodule, then the (operator-theoretic) \emph{adjoint} of $X$ is the operator $B$-$A$ bimodule $X^*$ defined as follows. As a vector space, $X^*$ is the complex conjugate of $X$ (in particular, $X^*$ does \emph{not} denote the Banach space dual of $X$). The bimodule structure on $X^*$ is defined by 
\begin{equation}\label{eq:opadj_bimod} 
b\cdot x^* \cdot a \coloneqq (a^*xb^*)^*,
\end{equation}
where $x\mapsto x^*$ is the canonical anti-linear isomorphism $X\to X^*$. For each $n$ one defines a norm on $M_n(X^*)$ by composing the given norm on $M_n(X)$ with the conjugate-transpose map $M_n(X^*)\to M_n(X)$; these norms make $X^*$ into an operator $B$-$A$ bimodule. See \cite[1.2.25 \& 3.1.16]{BLM} for details; note that \cite{BLM} uses a superscript $\star$ where we use $*$.

\subsection{$\boldsymbol *$-Representations as operator modules}\label{subsec:hmod}

Let $H$ be a (nonzero) complex Hilbert space. Fix a unit vector $h_0\in H$, and embed $H$ isometrically as a closed subspace of $\Bounded(H)$ by sending $h\in H$ to the rank-one operator $h_1\mapsto \langle h_0 | h_1\rangle h$.
The \emph{column operator space} structure on $H$, denoted by $H_c$, is the one induced by the above embedding; this structure is independent of the choice of unit vector $h_0$, up to completely isometric isomorphism. Every bounded operator $H\to K$ is completely bounded, with the same norm, when considered as a map $H_c\to K_c$. See \cite[1.2.23]{BLM}.

We shall also consider the \emph{row operator space} structure $H_r$ on $H$, defined (up to completely isometric isomorphism) by $H_r \coloneqq \left( \ol{H}_c\right)^*$, where $\ol{H}$ denotes the complex conjugate Hilbert space.

If $A\to \Bounded(H)$ is a $*$-representation of a $C^*$-algebra $A$ on $H$, then $H_c$ is a left operator module over $A$. The category $\hmod(A)$ of nondegenerate $*$-representations of $A$ (and bounded $A$-linear maps) thus embeds as a full subcategory of $\opmod(A)$. If $A$ is a type $\mathrm{I}$ $C^*$-algebra, then every (topologically) irreducible left operator $A$-module is completely isometrically isomorphic to an irreducible $*$-representation; see \cite[Proposition 2.4]{Crisp-Higson}. (We do not know whether the same is true when $A$ is not type $\mathrm{I}$.)

\subsection{Hilbert $\boldsymbol{C^*}$-modules and correspondences}\label{subsec:corresp}

Recall (from \cite[Chapter 8]{BLM} for instance) that a (left/right) \emph{Hilbert $C^*$-module} over a $C^*$-algebra $B$ is a nondegenerate (left/right) Banach $B$-module whose norm is induced by a positive-definite $B$-valued Hermitian  inner product $\langle\,\cdot\, |\,\cdot\,\rangle$. If $F$ is a Hilbert $C^*$-module over $B$, then $M_n(F)$ is in a natural way a Hilbert $C^*$-module over $M_n(B)$, and the induced norms on the spaces $M_n(F)$ make $F$ into an operator space and an operator $B$-module. If $H$ is a Hilbert space, considered as a right Hilbert $C^*$-module over $\C$, then the operator space structure defined on $H$ in this way is completely isometric to the column space structure $H_c$.

A map $t:F\to E$ of   Hilbert $C^*$-modules over $B$ is \emph{adjointable} if there is a map $s:E\to F$ satisfying $\langle t(f) | e\rangle = \langle f| s(e)\rangle$ for all $f\in F$ and all $e\in E$. Adjointable maps are automatically $B$-linear and completely bounded. In particular, the category $\cmod(B)$ of {left} $C^*$-modules over $B$ (and adjointable maps) embeds as a subcategory of $\opmod(B)$. 

In this paper we shall chiefly be concerned with \emph{right} Hilbert $C^*$-modules. We write $\Adjointable_B(F)$ for the $C^*$-algebra of adjointable operators on such a module $F$, and $\Compact_B(F)$ for the closed, two-sided ideal of \emph{$B$-compact} operators: the latter algebra is, by definition, the closed linear span of the operators 
\begin{equation}\label{eq:rk1_op}
f_1\otimes f_2^* : F\to F, \qquad f\mapsto f_1\langle f_2 | f\rangle 
\end{equation}
for $f_1,f_2\in F$.

Now suppose that $A$ is a second $C^*$-algebra. A \emph{$C^*$-correspondence} from $A$ to $B$ is a right Hilbert $C^*$-module $F$ over $B$, equipped with a $*$-homomorphism $A\to \Adjointable_B(F)$ making $F$ into a nondegenerate left $A$-module. The operator space structure on $F$ defined above makes $F$ into an operator $A$-$B$ bimodule. The adjoint operator space $F^*$ is a $B$-$A$ operator bimodule, but it is not in general a $C^*$-correspondence. (It is a \emph{left} $C^*$-module over $B$, but not a right $C^*$-module over $A$.)

\subsection{The Haagerup tensor product}\label{subsec:haag}

Let $A$ and $B$ be $C^*$-algebras, let $F$ be an operator $A$-$B$ bimodule, and let $X$ be a left operator $B$-module. A bilinear map $u:F\times X\to Z$ into an operator space $Z$ is said to be \emph{$B$-balanced} if $u(fb,x)=u(f,bx)$ for every $f\in F$, $b\in B$ and $x\in X$, while $u$ is said to be \emph{completely contractive} if, for every $n\geq 1$, the bilinear map
\[ 
u_n: M_n(F)\times M_n(X)\to M_n(Z), \qquad \left(u_n(f,x)\right)_{i,j} = \sum_{k} u(f_{i,k},x_{k,j})
\]
satisfies $\|u_n(f,x)\|_n \leq \|f\|_n \cdot \|x\|_n$. Composition of bounded Hilbert space operators is an example of a completely contractive bilinear map, and essentially the only example: see \cite[Theorem 1.5.7]{BLM}. 
 
The \emph{Haagerup tensor product} $F\otimes^h_B X$ is a left operator $A$-module, which is a completed quotient of the algebraic (balanced) tensor product $F\otimes^\alg_B X$, and which is characterised by the property that for each completely contractive, $B$-balanced bilinear map $u:F\times X\to Z$, the linear map 
\[
F\otimes^\alg_B X \to Z, \qquad f\otimes x \mapsto u(f,x) 
\]
extends to a completely contractive map of operator spaces $F\otimes^h_B X\to Z$.  See \cite[Sections 1.5 \& 3.4]{BLM}.

If $X$ is a left operator $B$-module, then the action map 
\[
B\otimes^{\alg}_B X\to X   \qquad b\otimes x\mapsto bx 
\]
extends to a completely isometric map $B\otimes^h_B X\to X$, whose image is precisely the nondegenerate part $B\cdot X$ of $X$.

The Haagerup tensor product is functorial: for each completely bounded module map $t\in \CB_B(X,Y)$, the map
\[
\id_F\otimes t : F\otimes^\alg_B X \to F\otimes^\alg_B Y,\qquad f\otimes x \mapsto f\otimes t(x ) 
\]
extends to a completely bounded $A$-module map $F\otimes^h_B X \to F\otimes^h_B Y$, 
with 
\[
\| \id_F\otimes t\|_{\cb}\leq \|t\|_{\cb}. 
\]
See \cite[3.4.5]{BLM}. 
It follows from this that the Haagerup tensor product with $F$ defines a functor 
\[
\functor{F}:\opmod(B)\to \opmod(A) \qquad X\mapsto F\otimes^h_B X,
\]
whose maps on morphism spaces are linear and contractive. These maps are in fact completely contractive, as is already well-known (this is asserted in \cite[8.2.19]{BLM}, for instance). Since we were not able to find a proof of this fact in the literature, we have included one as an appendix to this paper.

The tensor product functors 
$\functor{F}$ and $\functor{F}^*$ 
associated to a $C^*$-correspondence $F$ and to its adjoint operator bimodule $F^*$ will be the main objects of study in this paper. Here is a simple example.

\begin{example}\label{ex:predual}
Let $A\to \Bounded(H)$ be a nondegenerate $*$-representation, and consider $H$ as a $C^*$-correspondence from $A$ to $\C$. The functor 
\[
\functor{H}:\opmod(\C)\to \opmod(A)\qquad X\mapsto H_c\otimes^h X
\]
sends the row operator space $\overline{H}_r$, for example, to the $C^*$-algebra $\Compact_{\C}(H)$ of compact operators on $H$ (which is an $A$-module in the obvious way). 
The operator-theoretic adjoint 
\[
\functor{H}^*: \opmod(A)\to \opmod(\C)\qquad Y\mapsto \ol{H}_r\otimes^h_A Y
\]
sends the operator $A$-module $H_c$, for example, to the predual $\Bounded_A(H)_*$ of the von Neumann algebra $\Bounded_A(H)=A'\cap \Bounded(H)$. (Preduals of von Neumann algebras are operator spaces in a canonical way.) See \cite[1.4 \& 1.5.14]{BLM}.
\end{example}

The following two theorems of Blecher demonstrate that the functors $\functor{F}$ and $\functor{F}^*$, defined using the Haagerup tensor product of operator spaces, are also very relevant to the study of Hilbert-space representations and Hilbert $C^*$-modules:

\begin{theorem}\label{thm:Rieffel}\cite[Theorem 8.2.11]{BLM}.
Let $A$ and $B$ be $C^*$-algebras, and let $F$ be a $C^*$-correspondence from $A$ to $B$. For every Hilbert space representation $X\in\hmod(B)$, and every left Hilbert $C^*$-module $Y\in \cmod(A)$, the identity maps on the algebraic tensor products extend to completely isometric isomorphisms  
\[
F\otimes^h_B X \cong F\otimes^{C^*}_B X \quad \text{and}\quad F^*\otimes^h_A Y \cong F^*\otimes^{C^*}_A Y
\]
between the Haagerup tensor product and the (interior) Hilbert $C^*$-module tensor product of \cite{Rieffel_induced}.
 \hfill\qed
\end{theorem}

\begin{theorem} \label{thm:tensor_compact} \cite[Corollary 8.2.15]{BLM}.
Let $F$ be a $C^*$-correspondence from $A$ to $B$,  and let $F^*$ denote the adjoint operator $B$-$A$ bimodule. The map from
$F\otimes^\alg_B F^*$ to $\Adjointable_B(F)$ sending $f_1\otimes f_2^*$ to the operator \eqref{eq:rk1_op} extends to a completely isometric isomorphism $F\otimes^h_B F^* \cong \Compact_B(F)$. 
\hfill\qed
\end{theorem}

\subsection{Adjoint functors}\label{subsec:adj}

Turning now from operator theory to category theory, let us briefly review some terminology regarding adjoint functors. See \cite[Chapter IV]{MacLane} for details. (The historical discussion in \cite{MacLane_Stone} is also particularly germane to the present work.) We will assume that the reader is familiar with the language of categories, functors, and natural maps, as explained in \cite[Chapter I]{MacLane}.

\begin{definition} \label{def:adj}
Let $\category{A}$ and $\category{B}$ be categories. A functor $\functor{L}:\category{A}\to\category{B}$ is \emph{left-adjoint} to a functor $\functor{R}:\category{B}\to\category{A}$, and $\functor{R}$ is \emph{right-adjoint} to $\functor{L}$, if there are natural isomorphisms 
\begin{equation}\label{eq:adj_def}
\Hom_{\category{A}}(X, \functor{R}Y) \xrightarrow{\cong} \Hom_{\category{B}}(\functor{L}X, Y)
\end{equation}
for all objects $X\in \category{A}$ and $Y\in\category{B}$.
\end{definition}

The natural isomorphisms \eqref{eq:adj_def} are a priori isomorphisms of sets. In the cases we shall study, the $\Hom$-sets in the categories $\category{A}$ and $\category{B}$ will have some extra structure---they will be Banach spaces, for instance---and one would like to know whether the isomorphisms \eqref{eq:adj_def} preserve that structure. The following general fact about adjoint functors is useful in this regard. See \cite[IV.1]{MacLane} for a proof.

\begin{lemma}\label{lem:unit_counit}
Suppose that $\functor{L}:\category{A}\to \category{B}$ is left adjoint to $\functor{R}:\category{B}\to\category{A}$. There are natural maps 
\[
\eta_X : X \to \functor{R}\functor{L} X \qquad \text{and} \qquad \epsilon_Y : \functor{L}\functor{R} Y \to Y,
\]
for all objects $X\in \category{A}$ and $Y\in \category{B}$, such that the isomorphism \eqref{eq:adj_def} is the composition
\[
\Hom_{\category{A}}(X,\functor{R}Y) \xrightarrow{\functor{L}} \Hom_{\category{B}}(\functor{L}X, \functor{L}\functor{R} Y) \xrightarrow{\psi\mapsto \epsilon_Y\circ\psi} \Hom_{\category{A}}(\functor{L}X, Y),
\]
while the inverse of \eqref{eq:adj_def} is the composition
\[
\Hom_{\category{B}}(\functor{L}X, Y) \xrightarrow{\functor{R}} \Hom_{\category{A}}(\functor{R}\functor{L}X, \functor{R}Y) \xrightarrow{\phi\mapsto \phi\circ \eta_X} \Hom_{\category{A}}(X, \functor{R}Y). 
\]
\hfill\qed
\end{lemma}

The natural transformation $\eta$ is called the \emph{unit} of the adjunction, while $\epsilon$ is the \emph{counit}. Lemma \ref{lem:unit_counit} allows one to deduce properties of the isomorphisms \eqref{eq:adj_def} from properties of the functors $\functor{L}$ and $\functor{R}$. For example:

\begin{lemma}\label{lem:Banach}
Let $\category{A}$ and $\category{B}$ be \emph{Banach categories} (categories in which the $\Hom$ spaces are Banach spaces and the composition of morphisms is bilinear and continuous). Suppose that a functor $\functor{L}:\category{A}\to\category{B}$ is left-adjoint to a functor $\functor{R}:\category{B}\to \category{A}$. If one of the functors is a \emph{Banach functor} (meaning that the induced maps on $\Hom$ spaces are bounded and linear), then the adjunction isomorphisms \eqref{eq:adj_def} are isomorphisms of Banach spaces. 
\end{lemma}

\begin{proof}
By Lemma \ref{lem:unit_counit}, the isomorphism \eqref{eq:adj_def} is the composition
\[
\Hom_{\category{A}}(X,\functor{R}Y)\xrightarrow{\functor{L}} \Hom_{\category{B}}(\functor{L}X,\functor{L}\functor{R}Y) \xrightarrow{\epsilon_Y\circ} \Hom_{\category{B}}(\functor{L}X,Y).
\]
Since $\category{B}$ is a Banach category, the second arrow in this composition is bounded and linear. If $\functor{L}$ is a Banach functor then the first arrow in the composition is bounded and linear too, and then the open mapping theorem implies that \eqref{eq:adj_def} is an isomorphism of Banach spaces. A similar argument, using the unit of the adjunction, applies when $\functor{R}$ is assumed to be a Banach functor.
\end{proof}

\section{Adjoints to the Haagerup tensor product}\label{sec:adj}

Let $A$ and $B$ be $C^*$-algebras, and let $F$ be a $C^*$-correspondence from $A$ to $B$, considered as an operator $A$-$B$ bimodule. We denote by $F^*$ the adjoint operator $B$-$A$ bimodule.
Haagerup tensor product  with ${F}$ and ${F}^*$ gives functors 
\[
\functor{F}:\opmod(B)\xrightarrow{X\mapsto F\otimes^h_B X} \opmod(A) \quad \text{and}\quad \functor{F}^*:\opmod(A)\xrightarrow{Y\mapsto F^*\otimes^h_A Y}\opmod(B)
\]
on the categories on nondegenerate left operator modules over $A$ and $B$.

In \cite[Theorem 2.17]{Crisp-Higson} it was shown that if the image of the action homomorphism $A\to \Adjointable_B(F)$ is contained in $\Compact_B(F)$---in short, if $A$ acts on $F$ by $B$-compact operators---then the functor $\functor{F}^*$ is left-adjoint to the functor $\functor{F}$. In this section we prove a strong converse to this assertion:

\begin{theorem}\label{thm:adj_main}
Let $A$ and $B$ be $C^*$-algebras, let $F$ be a $C^*$-correspondence from $A$ to $B$, and let $F^*$ be the adjoint $B$-$A$ operator bimodule. The following are equivalent:
\begin{enumerate}[\rm(a)]
    \item  The functor $\functor{F}:\opmod(B)\xrightarrow{X\mapsto F\otimes^h_B X}\opmod(A)$ has a left adjoint.
    \item  The functor $\functor{F}^*: \opmod(A)\xrightarrow{Y\mapsto F^*\otimes^h_A Y} \opmod(B)$ has a \emph{strongly continuous} right adjoint.
    \item  The functor $\functor{F}$ is right-adjoint to the functor $\functor{F}^*$.
    \item  The action of $A$ on $F$ is by $B$-compact operators.
\end{enumerate}
If these conditions are satisfied, then the adjunction isomorphisms 
\[
\CB_B(X,\functor{F}Y) \cong \CB_A(\functor{F}^*X,Y)
\]
are completely contractive.
\end{theorem}

By definition, a functor $\functor{E}:\opmod(B)\to\opmod(A)$ is \emph{strongly continuous} if the maps 
\[
\functor{E}:\CB_B(X,Y) \to \CB_A(\functor{E}X,\functor{E}Y)
\]
are strong-operator continuous on bounded subsets. It is easy to see that the functor of Haagerup tensor product with an operator bimodule possesses this property. It may be possible to replace strong continuity in part (b) with some (a priori) weaker condition, but some sort of assumption on the adjoint is definitely needed, as the following example demonstrates.

\begin{example}
Let $F=H_c$ be a column Hilbert space, considered as a $C^*$-correspondence from $\C$ to $\C$. Then $F^* \cong \overline{H}_r$. For all operator spaces $X$ and $Y$ we have natural isomorphisms
\[
\CB(\overline{H}_r\otimes^h X, Y) \cong \CB(\overline{H}_r\widehat{\otimes} X, Y) \cong \CB(X, \CB(\overline{H}_r, Y)),
\]
where $\widehat{\otimes}$ is the operator space projective tensor product: the first isomorphism holds by \cite[Proposition 1.5.14(1)]{BLM} and the second holds by \cite[(1.50)]{BLM}. Thus the functor $\functor{F}^*:\opmod(\C)\to \opmod(\C)$ always has a right adjoint, while the  left  action of $\C$ on $H_c$ is by compact operators if and only if $H$ is finite-dimensional.
\end{example}

\begin{proof}[Proof of Theorem \ref{thm:adj_main}]
We noted above that the implication (d) implies (c) was proved in \cite{Crisp-Higson}. It was shown there that if $A$ acts on $F$ via a $*$-homomorphism $\alpha:A\to \Compact_B(F)$ then the natural maps 
\[
\eta_Y: Y \xrightarrow{\cong} A\otimes^h_A Y \xrightarrow{\alpha\otimes\id_Y}\Compact_B(F)\otimes^h_A Y \xrightarrow[\cong]{\text{Theorem \ref{thm:tensor_compact}}} F\otimes^h_B F^*\otimes^h_A Y
\]
and
\[
\epsilon_X : F^*\otimes^h_A F \otimes^h_B X \xrightarrow{\langle\,\cdot\, | \,\cdot\,\rangle\otimes \id_X} B\otimes^h_B X \xrightarrow{\cong} X
\]
are the unit and counit of an adjunction. The maps $\eta_Y$ are completely contractive, because $*$-homomorphisms are always completely contractive. The maps $\epsilon_X$ are also completely contractive, because the inner product on a Hilbert $C^*$-module can be expressed as a product of Hilbert space operators (see \cite[8.1.17]{BLM}), and the composition of operators is completely contractive. The functors $\functor{F}$ and $\functor{F}^*$ are also completely contractive, as shown in the appendix, and now Lemma \ref{lem:unit_counit} implies  that the adjunction isomorphisms are complete contractions (cf. the proof of Lemma \ref{lem:Banach}).

It is clear that (c) implies (a). Since Haagerup tensor product functors are strongly continuous, (c) implies  (b) too. So it remains to show that  (a) and (b) each imply (d).

Let us first show that (a) implies (d). Suppose that the functor $\functor{F}$ has a left adjoint $\functor{L}:\opmod(A)\to\opmod(B)$, so that we have natural isomorphisms 
\[
\CB_B(\functor{L}Y,X) \xrightarrow{\cong} \CB_A(Y, F\otimes^h_B X)
\]
for all $X\in\opmod(B)$ and all $Y\in \opmod(A)$. To save space we shall denote the operator $B$-module $\functor{L}A$ by $L$. Putting $X=B$ (recall that $F\otimes^h_B B\cong F$) and $Y=A$ into the above adjunction isomorphism then gives an isomorphism  
\[
\Phi: \CB_B(L,B) \to \CB_A(A,F), \qquad \Phi(t) = [\id_F\otimes t ]\circ \eta 
\]
for some $\eta=\eta_A\in\CB_A(A,F\otimes^h_B L)$ (see Lemma \ref{lem:unit_counit}). The categories $\opmod(A)$ and $\opmod(B)$ are Banach categories, and the functor $\functor{F}$ is a Banach functor, and so $\Phi$ is an isomorphism of Banach spaces (see Lemma \ref{lem:Banach}).

For each $f\in F$, consider the map 
\[
r_f\in \CB_A(A,F) \qquad r_f:a\mapsto \alpha(a)f,
\]
where $\alpha:A\to \Adjointable_B(F)$ denotes the action map. Let us write 
\[
t_f\coloneqq \Phi^{-1}(r_f)\in \CB_B(L,B).
\] 

Fix $a\in A$, and let $\{k_\lambda\}$ be a bounded approximate unit for the $C^*$-algebra $\Compact_B(F)$ of $B$-compact operators on $F$. For each $f\in F$ and each $\lambda$ we have an estimate
\[
\begin{aligned}
  \| (\alpha(a)-k_\lambda\circ \alpha(a) )f\| & = \| r_f(a) - k_\lambda\circ r_f(a)\| \\
  &  =  \| [\id_F\otimes t_f]\circ \eta(a) - k_\lambda\circ [\id_F\otimes t_f]\circ \eta(a) \| \\
 & = \| [\id_F\otimes t_f]\circ \eta(a) - [\id_F\otimes t_f]\circ [k_\lambda\otimes \id_{L}]\circ \eta(a) \| \\
 & \leq \|\id_F\otimes t_f\| \cdot \| \eta(a) - [k_\lambda\otimes \id_{L}]\circ \eta(a) \| \\
 & \leq \|\Phi^{-1}\| \cdot \|r_f\| \cdot \| \eta(a) - [k_\lambda\otimes \id_{L}]\circ \eta(a) \| \\
 &= \left( \|\Phi^{-1}\| \cdot  \| \eta(a) - [k_\lambda\otimes \id_{L}]\circ \eta(a) \| \right) \cdot \|f\|.
 \end{aligned}
\]
Now $F$ is nondegenerate as a left $\Compact_B(F)$-module, and so $F\otimes^h_B L$ is also nondegenerate as a left $\Compact_B(F)$-module. Thus 
\[
\| \eta(a) - [k_\lambda\otimes \id_{L}]\circ \eta(a) \| \xrightarrow[\ \lambda\to\infty\ ]{} 0.
\]
The computation above shows, therefore, that the operator norm
\[ 
\| \alpha(a)-k_\lambda\circ \alpha(a)\|  \xrightarrow[\ \lambda\to\infty\ ]{} 0.
\] 
Since $\alpha(a)\in \Adjointable_B(F)$, and $k_\lambda$ lies in the closed two-sided ideal $\Compact_B(F)\subseteq \Adjointable_B(F)$, the above convergence implies that $\alpha(a)\in \Compact_B(F)$ as claimed.

Now let us show that (b) implies (d). Suppose that the functor $F^*\otimes^h_A$ has a strongly continuous right adjoint $\functor{R}:\opmod(B)\to\opmod(A)$, so that we have natural isomorphisms (of Banach spaces, by Lemma \ref{lem:Banach})
\[
\CB_B(F^*\otimes^h_A Y, X) \xrightarrow{\cong} \CB_A(Y, \functor{R}X)
\]
for all $X\in\opmod(B)$ and all $Y\in \opmod(A)$. Putting $X=F^*$ and $Y=A$, and using the canonical isomorphism $F^*\otimes^h_A A\cong F^*$, we get an isomorphism
\[
\Psi: \CB_B(F^*, F^*) \to \CB_A(A, \functor{R}F^*).
\]
Naturality implies that this is an isomorphism of $A$-$\Compact_B(F)$ bimodules, where the bimodule structure is given on the left-hand side by 
\begin{equation}\label{eq:Fstar_adj1}
a\cdot t \cdot k: h\mapsto r_k\circ t(ha) \qquad (t\in \CB_B(F^*, F^*),\ a\in A,\ k\in \Compact_B(F),\ h\in F^*) 
\end{equation}
and on the right-hand side by 
\begin{equation*}\label{eq:Fstar_adj2}
a \cdot s \cdot k :a'\mapsto \functor{R}(r_k)(s(a'a)) \qquad (s\in \CB_A(A,\functor{R}F^*),\ a,a'\in A,\ k\in \Compact_B(F)),
\end{equation*}
where $r_k\in \CB_B(F^*, F^*)$ denotes right action of $k$ on $F^*$. 

Being an isomorphism of left $A$-modules, $\Psi$ restricts to an isomorphism between the $A$-nondegenerate submodules: 
\[
\Psi: A\cdot \CB_B(F^*, F^*) \xrightarrow{\cong} A\cdot \CB_A(A,\functor{R}F^*).
\]
Now, $A\cdot \CB_A(A,\functor{R}F^*)$ is isomorphic to $\functor{R}F^*$ as an $A$-$\Compact_B(F)$-bimodule (see \cite[Lemma 3.5.4]{BLM}). Since $F^*$ is nondegenerate as a right $\Compact_B(F)$-module, and $\functor{R}$ is strongly continuous, it follows that $\functor{R}F^*$ is also nondegenerate as a right $\Compact_B(F)$-module: indeed,   the bounded approximate unit $\{k_\lambda\}\subset \Compact_B(F)$ converges point-norm to the identity on $F^*$ (by the nondegeneracy of $F^*$), and then the strong continuity of $\functor{R}$ ensures that the same is true for the action of $\{k_\lambda\}$ on $\functor{R}F^*$. Pulling back by $\Psi$, we conclude that $A\cdot \CB_B(F^*, F^*)$ is nondegenerate as a right $\Compact_B(F)$-module.

This implies that for every $a\in A$ we have norm convergence
\begin{equation}\label{eq:ak_cvgce}
(a\cdot \id_{F^*})\cdot k_\lambda \xrightarrow[\ \lambda\to\infty\ ]{} a\cdot \id_{F^*} 
\end{equation}
inside $\CB_B(F^*, F^*)$. Consulting the definition \eqref{eq:Fstar_adj1} of the bimodule structure on $\CB_B(F^*,F^*)$, and the definition \eqref{eq:opadj_bimod} of the bimodule structure on $F^*$, we see that \eqref{eq:ak_cvgce} is equivalent to the assertion that 
\[
k_\lambda^*  \alpha(a^*) \xrightarrow[\ \lambda\to\infty\ ]{} \alpha(a^*)
\]
in the operator norm on $\Adjointable_B(F)$. Since the operators $k_\lambda^*$ all lie in the closed ideal $\Compact_B(F)$ of $\Adjointable_B(F)$, we conclude that the image of $\alpha$ is contained in $\Compact_B(F)$. This concludes the proof that (b) implies (d).
\end{proof}

\section{Applications}\label{sec:cor}

\subsection{Extension of scalars} 

\begin{corollary}\label{cor:pullback}
Let $A$ and $B$ be $C^*$-algebras, and let $\phi:A\to M(B)$ be a nondegenerate $C^*$-homomorphism from $A$ to the multiplier algebra of $B$. The pullback functor 
$\phi^*:\opmod(B)\to \opmod(A)$ admits a left adjoint if and only if $\phi(A)\subseteq B$.
\end{corollary}

\begin{proof}
The pullback functor can be identified with the Haagerup tensor product 
\[
Y \mapsto B\otimes^h_B Y ,
\]
where $B$ is a right $C^*$-module over $B$ in the obvious way, and a left $A$-module via the homomorphism $\phi$. The conclusion follows from Theorem \ref{thm:adj_main} and the well-known identification $\Compact_B(B)=B$.
\end{proof}

As a special case of the previous result (putting $A=\C$):

\begin{corollary}
Let $B$ be a $C^*$-algebra. The forgetful functor $\opmod(B)\to\opmod(\C)$ from operator $B$-modules to operator spaces admits a left adjoint if and only if $B$ is unital. The left adjoint, if it exists, is given by $X\mapsto B\otimes^h X$. \hfill\qed
\end{corollary}

\subsection{Frobenius reciprocity for locally compact groups}

\subsubsection{Operator space representations}

Let $G$ be a locally compact group, and let $H\subseteq G$ be a closed subgroup.  Write $\opmod(G)$ for the category $\opmod(C^*(G))$ of nondegenerate left operator $C^*(G)$-modules. Concretely, an object in $\opmod(G)$ is a norm-closed linear subspace $X\subseteq \Bounded(K)$, where $K$ is a unitary representation of $G$, such that $g\circ x\in X$ for every $g\in G$ and every $x\in X$. Recall that $\opmod(G)$ contains the category $\URep(G)\cong \hmod(C^*(G))$ of unitary representations as a full subcategory.

The restriction functor $\Res^G_H$ on unitary representations extends to a functor on operator modules in the obvious way:
\[
\Res^G_H: \opmod(G)\to \opmod(H).
\]
Rieffel has shown that Mackey's unitary induction functor $\Ind_H^G$ is equivalent to the functor of tensor product (in the sense of Hilbert $C^*$-modules) with a certain $C^*$-correspondence $I_H^G$ from $C^*(G)$ to $C^*(H)$: see \cite[Theorem 5.12]{Rieffel_induced}. Replacing Rieffel's tensor product by the Haagerup tensor product (cf. Theorem \ref{thm:Rieffel}), we obtain an extension of $\Ind_H^G$ to operator modules:
\[
\Ind_H^G:\opmod(H)\to \opmod(G) \qquad X\mapsto I_H^G \otimes^h_{C^*(H)} X.
\]

\begin{corollary}\label{cor:frob_op}
Let $G$ be a locally compact group, and let $H$ be a closed subgroup. 
\begin{enumerate}[\rm(a)]
\item The restriction functor 
$
\Res^G_H:\opmod(G)\to \opmod(H)
$
admits a left adjoint if and only if $H$ is open in $G$. The left adjoint, if it exists, is given by 
\[
X \mapsto C^*(G)\otimes^h_{C^*(H)} X.
\]
\item The unitary induction functor 
$
\Ind_H^G : \opmod(H)\to \opmod(G)
$
admits a left adjoint if and only if $H$ is cocompact in $G$. The left adjoint, if it exists, is given by 
\[
Y \mapsto \left(I_H^G\right)^* \otimes^h_{C^*(G)} Y. 
\]
\end{enumerate}
\end{corollary}

\begin{proof}
The restriction functor is (naturally isomorphic to) the pullback along the canonical inclusion $C^*(H)\to M(C^*(G))$. The image of this inclusion is contained in $C^*(G)$ if and only if $H$ is open in $G$, by \cite[Theorem 5.1]{Bekka}, and so part (a) follows from Corollary \ref{cor:pullback}.

Turning to the induction functor, Rieffel \cite{Rieffel_induced} established an isomorphism of $C^*$-algebras
\[
\Compact_{C^*(H)}(I_H^G) \cong C_0(G/H)\rtimes G,
\]
such that the action map $C^*(G)\to \Adjointable_{C^*(H)}(I_H^G)= M(\Compact_{C^*(H)}(I_H^G))$ corresponds to the canonical inclusion of $C^*(G)$ into the multiplier algebra of the crossed product. The image of the latter inclusion is contained in the crossed product itself if and only if the algebra $C_0(G/H)$ is unital---which is to say, if and only if $G/H$ is compact---and so part (b)  follows from Theorem \ref{thm:adj_main}.
\end{proof}

\begin{example}
To revisit the example considered in the introduction, let $H$ be the trivial subgroup, and consider the induction functor 
\[
\Ind_1^G:\opmod(\C)\to \opmod(G) \qquad X\mapsto L^2(G)_c\otimes^h X 
\]
from operator spaces to operator-space representations of $G$. The operator-theoretic adjoint 
\[
\left(\Ind_1^G\right)^*: \opmod(G)\to \opmod(\C) \qquad Y \mapsto \overline{L^2(G)}_r \otimes^h_{C^*(G)} Y
\]
sends the regular representation $L^2(G)$ to the Fourier algebra $A(G)$, as in Example \ref{ex:predual}. This functor is left-adjoint to $\Ind_1^G$ in the categorical sense if and only if $G$ is compact, but it is an interesting functor even when it is not a genuine adjoint.
\end{example}
 
\subsubsection{Unitary representations}\label{subsec:unitary_frob}

Let us contrast the above results with what happens in the setting of unitary representations.

\begin{lemma}\label{lem:unitary_frob}
Let $G$ be a locally compact group and let $H$ be a closed subgroup. 
\begin{enumerate}[\rm(a)]
\item If the functor $\Ind_H^G:\URep(H)\to \URep(G)$ has an adjoint, on either side, then the quasiregular representation $L^2(G/H)$ is isomorphic to a finite direct sum of irreducible representations of $G$.
\item If the functor $\Res^G_H:\URep(G)\to \URep(H)$ has an adjoint, on either side, then one has 
\[
\sum_{Y\in \widehat{G}} \dim Y^H <\infty
\]
where $\widehat{G}$ denotes the set of isomorphism classes of irreducible unitary representations of $G$, and $Y^H$ denotes the space of $H$-fixed vectors in $Y$.
\end{enumerate}
\end{lemma}

\begin{proof}
Suppose that $\Ind_H^G$ has a right adjoint $\functor{R}$, so that we have natural isomorphisms 
\begin{equation}\label{eq:unitary_frob_ind}
\Bounded_H(X, \functor{R} Y) \xrightarrow{\cong} \Bounded_G(\Ind_H^G X, Y) 
\end{equation}
for all representations $X\in \URep(H)$ and $Y\in \URep(G)$. Lemma \ref{lem:Banach} implies that \eqref{eq:unitary_frob_ind} is an isomorphism of Banach spaces. 

 Applying \eqref{eq:unitary_frob_ind} to the trivial one-dimensional representation $X=\C_H$ of $H$, and the quasi-regular representation $Y=L^2(G/H)\cong \Ind_H^G \C_H$ of $G$, we get an isomorphism of Banach spaces 
\[
\Bounded_H(\C_H, \functor{R}L^2(G/H) )\cong \Bounded_G(L^2(G/H), L^2(G/H)).
\]
The left-hand side is isomorphic, as a Banach space, to the Hilbert space of $H$-fixed vectors in $\functor{R}L^2(G/H)$, and is in particular reflexive. So the intertwining von Neumann algebra $\Bounded_G(L^2(G/H),L^2(G/H))$ is reflexive as a Banach space, and so is finite-dimensional. Thus $L^2(G/H)$ is a finite direct sum of irreducibles. If $\Ind_H^G$ has a left adjoint $\functor{L}$, then we get an isomorphism 
\[
\Bounded_H(\functor{L}L^2(G/H),\C_H) \cong \Bounded_G(L^2(G/H), L^2(G/H)).
\]
The left-hand side is again a Hilbert space, and so the same argument as above shows that $L^2(G/H)$ is a finite sum of irreducibles.

Now suppose that $\Res^G_H$ has a right adjoint $\functor{R}:\URep(H)\to \URep(G)$. We then have an isomorphism (again, of Banach spaces)
\[
\Bounded_H(\Res^G_H\functor{R}\C_H,\C_H)\cong \Bounded_G(\functor{R}\C_H, \functor{R}\C_H)  
\]
in which the left-hand side is a Hilbert space and the right-hand side is a von Neumann algebra. Hence, by the same reflexivity argument as above,  $\functor{R}\C_H$ is finite sum of irreducible representations of $G$. For every irreducible $Y\in \widehat{G}$ we have 
\[
\dim \Bounded_G(Y, \functor{R}\C_H) = \dim \Bounded_H(\Res^G_H Y, \C_H) = \dim Y^H,
\]
and since $\functor{R}\C_H$ is a finite sum of irreducibles we conclude that $\sum \dim Y^H$ is finite.
\end{proof}

\begin{corollary}\label{cor:unitary_frob}
Let $H$ be a closed subgroup of infinite index in a locally compact group $G$. If $G$ is compact, or if $H$ is normal, then neither $\Ind_H^G$ nor $\Res^G_H$ admits an adjoint, on either side, as functors on unitary representations.
\end{corollary}

\begin{proof}
If $G$ is compact then the Peter-Weyl theorem gives an isomorphism 
\[
L^2(G/H) \cong \bigoplus_{Y\in \widehat{G}} Y \otimes \ol{Y}^H,
\] 
where each $Y$ is finite-dimensional. Since $[G:H]=\infty$, $L^2(G/H)$ is infinite-dimensional and so the above sum must contain infinitely many nonzero terms. Now Lemma \ref{lem:unitary_frob} implies that neither $\Ind_H^G$ nor $\Res^G_H$ has an adjoint.

If $H$ is normal in $G$ then the infinite group $G/H$ has an infinite-dimensional von Neumann algebra $\Bounded_{G/H}(L^2(G/H))= \Bounded_G (L^2(G/H))$. So $\Ind_H^G$ cannot have an adjoint. Moreover, we have an inclusion $\widehat{G/H}\subseteq \widehat{G}$, given by viewing representations of $G/H$ as representations of $G$ on which $H$ acts trivially, and the Gelfand-Raikov theorem implies that 
\[
\sum_{Y\in \widehat{G}} \dim Y^H \geq \sum_{Z\in \widehat{G/H}} \dim Z =\infty.
\]
So $\Res^G_H$ doesn't have an adjoint.
\end{proof}

\begin{remark}
We do not know of any example of an infinite-index subgroup $H\subset G$ for which $\Ind_H^G$ or $\Res^G_H$ has an adjoint on unitary representations.
\end{remark}

\subsection{Adjoint correspondences}

Returning to the setting of arbitrary $C^*$-algebras $A$ and $B$, let $F$ be a $C^*$-correspondence from $A$ to $B$, and let $E$ be a $C^*$-correspondence from $B$ to $A$. The functors
\[
\functor{F}:\opmod(B)\xrightarrow{X\mapsto F\otimes^h_B X} \opmod(A) \quad \text{and}\quad \functor{E}:\opmod(A) \xrightarrow{Y\mapsto E\otimes^h_A Y} \opmod(B)
\]
restrict, by Theorem \ref{thm:Rieffel}, to functors
\[
\functor{F}_{\hmod}:\hmod(B)\to \hmod(A)\quad \text{and}\quad \functor{E}_{\hmod}:\hmod(A)\to \hmod(B)
\]
On the other hand, the Haagerup tensor product functors $\functor{F}^*$ and $\functor{E}^*$ associated to the adjoint bimodules $F^*$ and $E^*$ restrict to functors
\[
\functor{F}^*_{\cmod}:\cmod(A)\to \cmod(B) \quad \text{and}\quad \functor{E}^*_{\cmod}:\cmod(B)\to \cmod(A).
\]

Combining Theorem \ref{thm:adj_main} with results from \cite{KPW} and \cite{CCH2}, we have:

\begin{corollary}\label{cor:adjcorresp}
\begin{enumerate}[\rm(a)]
\item The functors $\functor{E}$ and $\functor{F}$ are two-sided adjoints if and only if the functors $\functor{E}^*_{\cmod}$ and $\functor{F}^*_{\cmod}$ are two-sided adjoints. 
\item If $\functor{E}$ is left-adjoint or right-adjoint to $\functor{F}$, then $\functor{E}_{\hmod}$ and $\functor{F}_{\hmod}$ are two-sided adjoints.
\end{enumerate}
\end{corollary}

\begin{proof}
Theorems of Kajiwara, Pinzari and Watatani \cite[Theorem 4.4(2), Theorem 4.13]{KPW}, as reformulated in \cite[Theorem 2.24]{CCH2}, imply that $\functor{E}^*_{\cmod}$ and $\functor{F}^*_{\cmod}$ are two-sided adjoints if and only if 
\begin{enumerate}[\rm(1)]
\item $E\cong F^*$ as operator $B$-$A$ bimodules, and
\item The action of $A$ on $F$ is by $B$-compact operators, and the action of $B$ on $E$ is by $A$-compact operators. 
\end{enumerate}
It was shown in \cite[Theorem 3.15]{CCH2} that the condition (1) alone implies that $\functor{F}_{\hmod}$ and $\functor{E}_{\hmod}$ are two-sided adjoints.

By Theorem \ref{thm:adj_main}, the conditions (1) and (2) together are also equivalent to the functors $\functor{E}$ and $\functor{F}$ being two-sided adjoints, and so part (a) is proved.

If $\functor{E}$ is left-adjoint to $\functor{F}$, then condition (1) holds. Indeed, Theorem \ref{thm:adj_main} implies that $\functor{F}^*$ is also left-adjoint to $\functor{F}$, and by the uniqueness of adjoint functors there is a natural isomorphism $\functor{E}\cong \functor{F}^*$; in particular there is a completely bounded isomorphism 
\[
E\cong \functor{E}(A) \cong \functor{F}^*(A) \cong F^*,
\]
which by naturality is an isomorphism of operator $B$-$A$ bimodules. The same argument applies if we assume $\functor{E}$ to be right-adjoint to $\functor{F}$. This proves part (b).
\end{proof}

\subsection{Parabolic induction}

Let $G$ be a real reductive group, and let $P$ be a proper parabolic subgroup of $G$ with Levi decomposition $P=LN$. For example, take $G=\GL(n,R)$, let $P$ be the subgroup of upper-triangular matrices, let $L$ be the subgroup of diagonal matrices, and let $N$ be the subgroup of unipotent upper-triangular matrices.  Clare \cite{Clare} has constructed a $C^*$-correspondence $C^*_r(G/N)$, from $C^*_r(G)$ to $C^*_r(L)$,  whose associated tensor product functor
\[
\Ind_P^G:\hmod(C^*_r(L)) \to \hmod(C^*_r(G)) \qquad X\mapsto C^*_r(G/N)\otimes_{C^*_r(L)} X
\]
is (naturally isomorphic to) the well-known functor of \emph{parabolic induction} of tempered unitary representations. 

Together with Clare and Higson, we showed in \cite{CCH1} and \cite{CCH2} that the adjoint operator bimodule $C^*_r(N\backslash G) \coloneqq C^*_r(G/N)^*$ is actually a $C^*$-correspondence---its operator space structure is induced by a $C^*_r(G)$-valued inner product---and that the \emph{parabolic restriction} functor $\Res^G_P$ associated to the correspondence $C^*_r(N\backslash G)$ is a two-sided adjoint to $\Ind_P^G$ on the categories $\hmod$. In  \cite{Crisp-Higson} we studied the corresponding functors on operator modules, and we showed that $\Res^G_P$ is left-adjoint to $\Ind_P^G$ in this setting; unlike the results of \cite{CCH1} and \cite{CCH2}, the same is also true for operator modules over the \emph{full} group $C^*$-algebras. The `second adjoint theorem' of Bernstein \cite{Bernstein} asserts that in the related context of smooth representations of $p$-adic reductive groups, parabolic induction and restriction are adjoints on both sides. (Strictly speaking, the functor $\Ind_P^G$ has left adjoint $\Res^G_P$ and right adjoint $\Res_{\overline{P}}^G$, where $\overline{P}$ is the parabolic subgroup opposite to $P$.) It is thus natural to ask whether $\Ind_P^G$ has a right adjoint in the operator-module setting.

\begin{corollary}\label{cor:parabolic}
If $P=LN$ is a proper parabolic subgroup of a real reductive group $G$, then the functor $\Ind_P^G:\opmod(C^*_r(L) )\to \opmod(C^*_r(G) )$  has no strongly continuous right adjoint.
\end{corollary}

\begin{proof}
The explicit computations of the correspondences $C^*_r(G/N)$ and $C^*_r(N\backslash G)$ in \cite{CCH1} show that $C^*_r(L)$ does not act by $C^*_r(G)$-compact operators on $C^*_r(N\backslash G)$ (unless $L=P=G$). Now the corollary follows from Theorem \ref{thm:adj_main}.
\end{proof}

The construction in \cite[Section 4.3]{Crisp-Higson} hints at a possible second adjoint theorem for operator modules over suitably chosen non-self-adjoint subalgebras of the reduced $C^*$-algebras.

\section*{Appendix: The Haagerup tensor product is completely contractive}

Let $A$ and $B$ be $C^*$-algebras, let $F$ be an operator $A$-$B$ bimodule, let $X$ and $Y$ be left operator $B$-modules,  and consider the map
\begin{equation*}
\functor{F}(X,Y):\CB_B(X,Y)\to \CB_A(F\otimes^h_B X, F\otimes^h_B Y) \qquad t\mapsto \id_F\otimes t.
\end{equation*}

The following fact is well known, but we were not able to find a proof in the literature.

\begin{lemma*}
The map $\functor{F}(X,Y)$ is completely contractive. 
\end{lemma*}

\begin{proof}
The map is contractive, by \cite[Lemma 3.4.5]{BLM}. To show that it is completely contractive we need to consider the amplifications
\[
\functor{F}(X,Y)_n : \CB_B(X, M_n(Y)) \to \CB_A(F\otimes^h_B X, M_n(F\otimes^h_B Y))
\]
defined by 
\[
\functor{F}(X,Y)_n [t_{ij}] : f\otimes x \mapsto [f\otimes t_{ij}(x)].
\]

We claim that the map
\[
\psi: F\otimes^{\alg}_B M_n(Y) \to M_n(F\otimes^{\alg}_B Y) \qquad f\otimes [y_{ij}] \mapsto [f\otimes y_{ij}]
\]
extends to a completely contractive map 
\[
F\otimes^h_B M_n(Y) \to M_n(F\otimes^h_B Y).
\]
To see this, first ignore the $B$-balancing, and write $\psi$ as 
\[ 
\psi: F\otimes^\alg C_n \otimes^\alg Y \otimes^\alg R_n \to C_n \otimes^\alg F\otimes^\alg Y \otimes^\alg R_n, \qquad \psi = \operatorname{flip}\otimes \id\otimes \id.
\]
(Here $C_n=(\C^n)_c$ and $R_n=(\C_n)_r$ denote respectively the column- and the row-operator space structures on $\C^n$. Multiplying columns by rows gives an isomorphism $C_n\otimes^\alg V\otimes^\alg R_n\cong M_n(V)$ for each complex vector space $V$.)

The claim thus boils down to verifying that the flip $F\otimes C_n \to C_n \otimes F$ induces a complete contraction on the Haagerup tensor products. This is true, because
\begin{enumerate}[--]
\item The identity map on $F\otimes^\alg C_n$ induces a complete contraction from $F\otimes^h C_n$ to $F\otimes^{\mathrm{min}}C_n$, where $\otimes^{\mathrm{min}}$ is the minimal operator space tensor product; see \cite[Proposition 1.5.13]{BLM}.
\item The flip map on $\otimes^\alg$ induces a complete isometry on $\otimes^{\mathrm{min}}$; see \cite[1.5.1]{BLM}.
\item The identity on $C_n\otimes^\alg F$ extends to a complete isometry from $C_n\otimes^{\mathrm{min}} F$ to $C_n\otimes^h F$; see \cite[1.5.14(3)]{BLM}.
\end{enumerate}

Now the composition of the completely contractive map $\psi$ with the completely contractive quotient map $M_n(F\otimes^h Y)\to M_n(F\otimes^h_B Y)$ is a completely contractive, $B$-balanced map
\[
X\otimes^h M_n(Y) \to M_n(X\otimes^h_B Y).
\]
By the universal property of the Haagerup tensor product, the induced map
\[ X\otimes^h_B M_n(Y) \to M_n(X\otimes^h_B Y)\]
is also completely contractive.

Now $\functor{F}(X,Y)_n$ factors as the composition of two contractive maps, 
\begin{multline*}
\CB_B(X, M_n(Y)) \xrightarrow{\functor{F}(X,M_n(Y))} \CB_A(F\otimes^h_B X, F\otimes^h_B M_n(Y) )\\ \xrightarrow{\ t\mapsto \psi\circ t\ } \CB_A(F\otimes^h_B X, M_n(F\otimes^h_B Y)),
\end{multline*}
and so $\functor{F}(X,Y)_n$ is itself contractive.
\end{proof}

\bibliographystyle{alpha}
\bibliography{frob-haag}

\begin{thebibliography}{CCH16b}

\bibitem[Ber87]{Bernstein}
J.~N. Bernstein.
\newblock Second adjointness for representations of reductive $p$-adic groups.
\newblock Draft available at
  \url{http://www.math.uchicago.edu/~mitya/langlands.html}, 1987.

\bibitem[BKLS98]{Bekka}
M.~B. Bekka, E.~Kaniuth, A.~T. Lau, and G.~Schlichting.
\newblock Weak{$^*$}-closedness of subspaces of {F}ourier-{S}tieltjes algebras
  and weak{$^*$}-continuity of the restriction map.
\newblock {\em Trans. Amer. Math. Soc.}, 350(6):2277--2296, 1998.

\bibitem[Ble01]{Blecher}
D.~P. Blecher.
\newblock On {M}orita's fundamental theorem for {$C^\ast$}-algebras.
\newblock {\em Math. Scand.}, 88(1):137--153, 2001.

\bibitem[BLM04]{BLM}
D.~P. Blecher and C.~Le~Merdy.
\newblock {\em Operator algebras and their modules---an operator space
  approach}, volume~30 of {\em London Mathematical Society Monographs. New
  Series}.
\newblock The Clarendon Press, Oxford University Press, Oxford, 2004.
\newblock Oxford Science Publications.

\bibitem[CCH16a]{CCH2}
P.~Clare, T.~Crisp, and N.~Higson.
\newblock Adjoint functors between categories of {H}ilbert ${C}^{\ast
  }$-modules.
\newblock {\em Journal of the Institute of Mathematics of Jussieu},
  FirstView:1--36, 7 2016.

\bibitem[CCH16b]{CCH1}
P.~Clare, T.~Crisp, and N.~Higson.
\newblock Parabolic induction and restriction via ${C}^*$-algebras and
  {H}ilbert ${C}^*$-modules.
\newblock {\em Compositio Mathematica}, 152:1286--1318, 6 2016.

\bibitem[CH16]{Crisp-Higson}
T.~Crisp and N.~Higson.
\newblock Parabolic induction, categories of representations and operator
  spaces.
\newblock In {\em {O}perator algebras and their applications: a tribute to
  {R}ichard {V}. {K}adison}, volume 671 of {\em Contemp. Math.} Amer. Math.
  Soc., Providence, RI, 2016.

\bibitem[CH17]{CH-HC}
T.~Crisp and N.~Higson.
\newblock A second adjoint theorem for {SL(2,R)}.
\newblock To appear in \textit{Contemp. Math.}, 2017.

\bibitem[Cla13]{Clare}
P.~Clare.
\newblock Hilbert modules associated to parabolically induced representations.
\newblock {\em J. Operator Theory}, 69(2):483--509, 2013.

\bibitem[Cur99]{Curtis}
C.~W. Curtis.
\newblock {\em Pioneers of representation theory: {F}robenius, {B}urnside,
  {S}chur, and {B}rauer}, volume~15 of {\em History of Mathematics}.
\newblock American Mathematical Society, Providence, RI; London Mathematical
  Society, London, 1999.

\bibitem[Eym64]{Eymard}
P.~Eymard.
\newblock L'alg\`ebre de {F}ourier d'un groupe localement compact.
\newblock {\em Bull. Soc. Math. France}, 92:181--236, 1964.

\bibitem[Fel64]{Fell}
J.~M.~G. Fell.
\newblock Weak containment and induced representations of groups. {II}.
\newblock {\em Trans. Amer. Math. Soc.}, 110:424--447, 1964.

\bibitem[KPW04]{KPW}
T.~Kajiwara, C.~Pinzari, and Y.~Watatani.
\newblock Jones index theory for {H}ilbert {$C^*$}-bimodules and its
  equivalence with conjugation theory.
\newblock {\em J. Funct. Anal.}, 215(1):1--49, 2004.

\bibitem[Mac49]{Mackey_imprimitivity}
G.~W. Mackey.
\newblock Imprimitivity for representations of locally compact groups. {I}.
\newblock {\em Proc. Nat. Acad. Sci. U. S. A.}, 35:537--545, 1949.

\bibitem[Mac52]{Mackey_inducedI}
G.~W. Mackey.
\newblock Induced representations of locally compact groups. {I}.
\newblock {\em Ann. of Math. (2)}, 55:101--139, 1952.

\bibitem[Mac53]{Mackey_inducedII}
G.~W. Mackey.
\newblock Induced representations of locally compact groups. {II}. {T}he
  {F}robenius reciprocity theorem.
\newblock {\em Ann. of Math. (2)}, 58:193--221, 1953.

\bibitem[Mau52]{Mautner}
F.~I. Mautner.
\newblock Induced representations.
\newblock {\em Amer. J. Math.}, 74:737--758, 1952.

\bibitem[ML70]{MacLane_Stone}
S.~Mac~Lane.
\newblock The influence of {M}. {H}. stone on the origins of category theory.
\newblock In {\em Functional Analysis and Related Fields: Proceedings of a
  Conference in honor of Professor Marshall Stone, held at the University of
  Chicago, May 1968}, pages 228--241. Springer Berlin Heidelberg, 1970.

\bibitem[ML98]{MacLane}
S.~Mac~Lane.
\newblock {\em Categories for the working mathematician}, volume~5 of {\em
  Graduate Texts in Mathematics}.
\newblock Springer-Verlag, New York, second edition, 1998.

\bibitem[Moo62]{Moore}
C.~C. Moore.
\newblock On the {F}robenius reciprocity theorem for locally compact groups.
\newblock {\em Pacific J. Math.}, 12:359--365, 1962.

\bibitem[Pal74]{Palmquist}
P.~H. Palmquist.
\newblock Adjoint functors induced by adjoint linear transformations.
\newblock {\em Proc. Amer. Math Soc.}, 44:251--254, 1974.

\bibitem[Rie67]{Rieffel_Banach}
M.~A. Rieffel.
\newblock Induced {B}anach representations of {B}anach algebras and locally
  compact groups.
\newblock {\em J. Functional Analysis}, 1:443--491, 1967.

\bibitem[Rie74]{Rieffel_induced}
M.~A. Rieffel.
\newblock Induced representations of {$C^{\ast} $}-algebras.
\newblock {\em Advances in Math.}, 13:176--257, 1974.

\bibitem[Ros77]{Rosenberg}
J.~Rosenberg.
\newblock Frobenius reciprocity for square-integrable factor representations.
\newblock {\em Illinois J. Math.}, 21(4):818--825, 1977.

\bibitem[Wei40]{Weil}
A.~Weil.
\newblock {\em L'int\'egration dans les groupes topologiques et ses
  applications}.
\newblock Actual. Sci. Ind., no. 869. Hermann et Cie., Paris, 1940.

\end{thebibliography}

\end{document}